\documentclass[12pt]{amsart}
\usepackage{array}
\usepackage{amstext}
\usepackage{amsfonts}
\usepackage{amsmath}
\usepackage{amsthm}
\usepackage{filecontents}
\usepackage{enumerate}

\usepackage{amssymb}
\usepackage{tikz}
\usepackage[colorlinks=true, linkcolor=blue]{hyperref}

\usetikzlibrary{arrows, positioning}
\tikzset{Mylong/.style={text width=3.1cm, align=center}, myarr/.style={->, double equal sign distance, -implies}}

\usepackage[left = 3cm, right = 3cm, top = 3cm, bottom = 3cm]{geometry}

\newtheorem*{theorem*}{Theorem} 
\newtheorem{theorem}{Theorem}[section]
\newtheorem{corollary}[theorem]{Corollary}
\newtheorem*{corollary*}{Corollary}
\newtheorem{lemma}[theorem]{Lemma}
\newtheorem*{lemma*}{Lemma}
\newtheorem{proposition}[theorem]{Proposition}

\theoremstyle{definition}
\newtheorem{definition}[theorem]{Definition}
\newtheorem{example}[theorem]{Example}
\newtheorem*{example*}{Example}
\theoremstyle{remark}
\newtheorem*{remark}{Remark}

\title[On periodicity and hypercyclic weighted translation operators] {On periodicity and hypercyclic weighted translation operators}
\author[Kui-Yo Chen]{Kui-Yo Chen}


\address{Department of Mathematics, National Taiwan University, Taiwan}
\email{kuiyochen1230@gmail.com, r04221001@ntu.edu.tw}

\date{\today}
\begin{document}

\begin{abstract}
We use the generization of Weyl's equidistribution theorem to characterize several necessary conditions of hypercyclic weighted translation operators with periodic element.
\end{abstract}

\maketitle
\addcontentsline{toc}{section}{Title}
\baselineskip17pt

\section{Introduction} 
\label{sec:introduction}

The main purpose of this paper is to discuss some necessary conditions of the existence of hypercyclic weighted translation operators with a periodic element. And it lead us to Section \ref{sec:a_necessary_condition_of_hypercyclic_weighted_translation_operators} and Section \ref{sec:some_properties_of_hypercyclic_weighted_translations_on_p_adic}. Section \ref{sec:homogeneous_equidistribution_on_compact_groups} is prepared for the Section \ref{sec:a_necessary_condition_of_hypercyclic_weighted_translation_operators} which defines a property  {\it homogeneous equidistribution} (Definition \ref{deftransunifequi}) that is the uniform version of the well-known Weyl's equidistribution on compact group.

Let $G$ be a second countable locally compact Hausdorff group equipped with the right Haar measure and $a$ be an element of $G$. The weighted translation operator $T_{a,w}$ is a bounded linear self-map on the Banach space $L^p(G)$ for some $p \in [1, \infty)$ defined by 
\[T_{a,w}(f)(x):=w(x)f(xa^{-1}),\]
where the weight $w$ is a bounded continuous function from $G$ to $(0,\infty )$. And we denote $T_{a,1}$ by $T_a$. Notice that $T_a w$ is a function translated by $a$ while $T_{a,w}$ is a weighted translation operator.

In order to analyze $T_{a,w}$, we are used to classifying some different topological properties of the elements in $G$. We call an element $a$ of $G$ {\it torsion} if it has finite order. An element $a$ is {\it periodic} if the closed subgroup $G(a)$ generated by $a$ (i.e. $G(a)=\overline{<a>}$) is compact in $G$. An element $a$ is {\it aperiodic} if it is not periodic.

An operator $T$ on a Banach space $X$ is called {\it hypercyclic} if there exists a vector $x\in X$ such that its orbit is dense in X (i.e. $orb(T,x):=\left\{T^nx|n\in \mathbb{N}\right\}$ is dense in $X$).

\begin{lemma*} $ $
\cite[Lemma 1.1.]{Hypercyclic_on_groups}
Let $G$ be a locally compact group and let $a\in G$ be a torsion element.
Then any weighted translation $T_{a,w} : L^p(G) \to L^p(G)$ is not hypercyclic, for
$1 \le p <\infty$.
\end{lemma*}

\begin{theorem*} $ $
\cite[Theorem 3.2.]{On_aperiodicity_and_hypercyclic_weighted}
Let $G$ be a locally compact group and $a$ is an aperiodic element in $G$. Then there exists a weighted translation operator $T_{a,w}$ which is mixing, chaotic and frequently hypercyclic on $L^p(G)$ for all $p\in [1,\infty)$, simultaneously.
\end{theorem*}

\cite[Lemma 1.1]{Hypercyclic_on_groups} gives a non-existence of hypercyclic weighted translation when the element $a$ is a torsion.
On the other hand, \cite[Theorem 3.2.]{On_aperiodicity_and_hypercyclic_weighted} gives the general exsitence of hypercyclic weighted translation operators when $a$ is an aperiodic element.
So the question of the existence of hypercyclic weighted translation focus on the case $a$ is non-torsion periodic.
Notice that this question is still an open problem.
To solve this problem, we discussed some necessary conditions of the existence of hypercyclic weighted translation operators with a periodic element. Thus, we got the following two main results.
\begin{theorem*}(Theorem \ref{conditionofwsoncompactgp})
Let $G$ be a compact Hausdorff group, and let $a$ be an element in $G$ such that $G=\overline{<a>}$. If $T_{a,w}$ is hypercyclic on $L^p(G)$ for some $p\in [1,\infty )$, then 
\[\int_G \ln w = 0.\]
\end{theorem*}

\begin{theorem*}(Theorem \ref{p_adic_main})
If $T_{a,w}$ is hypercyclic on $L^{p'}(G)$ where $G=\mathbb{Q}_p \text{ or } \mathbb{Z}_p$ for some prime $p$ and $p'\in [1,\infty)$, then
$$\int_G \ln w = 0$$
and $w$ must not be a locally constant function.
\end{theorem*}

And also, there is an interesting and unanticipated corollary in Section \ref{sec:homogeneous_equidistribution_on_compact_groups}.

\begin{corollary*}(Corollary \ref{noncyclic})
Let $G$ be a finite group. Then the following are equivalent.
\begin{enumerate}[{(1)}]
\item
$G$ is non-cyclic.
\item
For any $a \in G$, there exists $\pi \in \hat{G}\setminus \{[1]\}$ such that $\pi(a)$ has an eigenvalue $1$.
\end{enumerate}
\end{corollary*}


\section{Notation and Preliminary} 
\label{sec:notation}
In this paper, we focus on the case that $G$ is a compact Hausdorff group except few abelian examples. In these case, the Haar measure of $G$ are two-side invariant.

We denote $\widetilde{U}$ to be an arbitrary borel set with $U$ being an open set and $U\subseteq \widetilde{U}\subseteq \overline{U}$
(See also \cite[Chapter 3]{Uniform_Distribution_of_Sequences}, they call $\widetilde{U}$ a ``$\mu$-continuity set'' when $|\partial U|:=\mu(\partial U)=0$).

\begin{example*}
Let $G=S^1=[0,1)$ and $U=(0,\frac{1}{2})$, where $S^1$ denotes the circle group. There are $4$ possibilities of the expression $\widetilde{U}$, $(0,\frac{1}{2})$, $[0,\frac{1}{2})$, $(0,\frac{1}{2}]$ and $[0,\frac{1}{2}]$.
\end{example*}

Note that this notation $\widetilde{U}$ might somewhat confuse $[0,\frac{1}{2})$ with $(0,\frac{1}{2}]$, for instance. But we would like to use it carefully and still denote those sets by the same symbol.


We denote $\mathcal{A}$ to be an algebra generated by the sets $\{\text{every }\widetilde{U}|\partial U \text{ has measure }0\}$, that is, the smallest set containing every $\widetilde{U}$ with $|\partial U|=0$ which is closed under union operation and complement operation.

\begin{example*}
In the case of $G=S^1=[0,1)$, these sets $(0,\frac{1}{2})$, $[0,\frac{1}{2})$, $(0,\frac{1}{2}]$ and $[0,\frac{1}{2}]$ belong to $\mathcal{A}$.
\end{example*}

Looking $\mathcal{A}$ more closely, we have the following.
\begin{proposition}
\label{3forms}
The elements in $\mathcal{A}$ are one of the forms below:
\begin{enumerate}[{(1)}]
\item
$\widetilde{U}$, where $\partial U$ has measure $0$.
\item
$S=\widetilde{U}\setminus U\subseteq \partial U$, where $\widetilde{U}$ has the form (1).
\item
$\widetilde{U}\cup \cup_{i=1}^k S_i$, where $\widetilde{U}$ has the form (1), and $S_i$ have the form (2).
\end{enumerate}
\end{proposition}
Note that these forms are not disjoint.
\begin{proof}
$ \\$
\noindent 1. We first consider the union and complement of form (1). It is easy to see that the union of form (1) is also form (1). But for the complement of form (1), there are two possibilities, form (1) or form (2), determined by whether it is dense in $G$.

\noindent 2. For the form (2), let $S\subseteq \partial U$ be an element of the form (2) and $K$ be form (1) or (2) or (3), where $\widetilde{U}$ has the form (1). Then $S \cup K$ has the form (3) and $S^c$ has the form (1), this follows by the facts that
$\partial S^c = \partial S \subseteq \overline{S}\subseteq \partial U$ is measure $0$ and $S^c=int(S^c)\cup (S^c\setminus int(S^c))$.

\noindent 3. Finally, we consider the form (3).
It is easy to see that the union of the elements of the form (3) are also the form (3).

Let $\widetilde{U}\cup \cup_{i=1}^k S_i$ be an element of the form (3). Then its complement is $\widetilde{U}^c\cap (\cap_{i=1}^k S_i^c)$, which is the finite intersection of the form (1). And the finite intersection of the form (1) is either the form (1) or the form (2).

In conclusion, the sets generated by the form (1) are also of the three forms.
\end{proof}

Recall the representation theory of compact groups \cite[Chapter 5]{A_Course_in_Abstract_Harmonic_Analysis}. We denote $\hat{G}$ to be the set of unitary equivalence classes of irreducible unitary representations of the compact Hausdorff $G$. Let $[\pi ]\in \hat{G}$. Then $\pi :G\to U(H_{\pi})\subset M(H_{\pi})$ is a group homomorphism that $H_{\pi}$ is denoted to be the finite dimension complex Hilbert space w.r.t. $\pi$, see \cite[Theorem 5.2]{A_Course_in_Abstract_Harmonic_Analysis}. $M(H_{\pi})$ is denoted to be the collection of linear operators on $H_{\pi}$ and $U(H_{\pi})$ is its subcollection of unitary matries. We denote $\pi _{i,j}(x):=<\pi(x)e_i,e_j>$ to be the matrix elements, where $\{e_i\}$ is the standard basis of $H_{\pi}$.

\subsection{Notation for p-adic analysis} 
\label{sec:notation_for_p_adic_analysis}
This subsection is prepared for the discussions which refer to the p-adic groups.

Let $\mathbb{Q}_p$ denote the p-adic field and let $\mathbb{Z}_p$ denote its ring of integer for some prime number $p$. Since they are both DVR (discrete valuation ring), we denote their valuation by $v_p$, (ex: $v_5(100)=v_5(5^2)+v_5(4)=2+0=2$). Hence the p-adic norm $|\cdot|_p:=p^{-v_p(\cdot)}$ defines their topology. (See \cite{A_Course_in_p-adic_Analysis}).

We denote $M_a$ as a multiplication operator which is defined by
\[M_a(f)(x):=f(ax).\]
\begin{remark}
In the discussions of the p-adic groups the weighted translation operator will become $T_{a,w}(f)(x):=w(x)f(x-a)$, since each p-adic group is an additive group. Both weighted translation operators and multiplication operators act on the $L^{p'}(G)$, where $G=\mathbb{Q}_p$ or $\mathbb{Z}_p$ and $p'\in [1,\infty)$.
Notice that the notation for ``$p$'' is different from other places in this paper.
\end{remark}

\section{Homogeneous equidistribution on compact groups} 
\label{sec:homogeneous_equidistribution_on_compact_groups}
$ $

Let $\mathbb{C}(G)$ denote the set of continuous complex value functions on $G$. And we assume that $|G|=1$.
\begin{lemma}
\label{uniformlyconverge}
Let $G$ be a compact Hausdorff group, and let $a$ be an element in $G$ such that $\pi(a)$ do not have nontrivial fixed points in $H_{\pi}$(equivalently, $I_{H_{\pi}}-\pi(a)$ is invertible) for all $[\pi ]\in \hat{G}\setminus \{[1]\}$. Then for any $f\in \mathbb{C}(G)$, define $g_{f,N}(x):=\frac{1}{N} \sum\limits_{n=1}^{N-1}f(xa^{-n})$. We have
\[g_{f,N}\text{ converges uniformly to the constant }\int_G f.\]
\end{lemma}
\begin{proof}
Let $S$ be the collection consisting all $f$ which satisfies this Lemma. It is non-empty, since all constant functions belong to $S$.

Recall that $span\left\{\pi _{i,j}|[\pi ]\in \hat{G} \text{ and $\pi _{i,j}$ be an matrix element of $\pi$.}\right\}$ is dense in $\mathbb{C}(G)$ \cite[Theorem 5.11]{A_Course_in_Abstract_Harmonic_Analysis}. In particular, the only closed subspace of $\mathbb{C}(G)$ containing all matrix elements is $\mathbb{C}(G)$ itself. So it is sufficient to show $S$ is a closed subspace containing all matrix elements.

Note that $\left\|g_{h,N}\right\|_{\infty}\le \left\|h\right\|_{\infty}$ for any $h\in \mathbb{C}(G)$.

\noindent\textbf{Claim}: $S$ is a closed subspace in $\mathbb{C}(G)$ (w.r.t. $\|\cdot\|_{\infty}$).

$S$ is a subspace, since for any $f,f'\in S$ and $A,B\in \mathbb{C}$ we have $g_{Af+Bf',N}=Ag_{f,N}+Bg_{f',N}$. So $g_{Af+Bf',N}$ also converges uniformly to the constant $\int_G Af+Bf'$.

Let $f_i\in S$ and $f_i$ converges uniformly to $f$. That means that for any $\varepsilon >0$, we have $\left\|f-f_{i_0}\right\|_{\infty}<\frac{\varepsilon }{3}$ for some ${i_0}$ and there is a number $N'$ such that $\left\|g_{f_{i_0},N}-\int_G f_{i_0}\right\|_{\infty}<\frac{\varepsilon }{3}$ for all $N>N'$. Then 
\begin{align*}
\left\|g_{f,N}-\int_G f\right\|_{\infty}
&\le \left\|g_{f_{i_0},N}-\int_G f_{i_0}\right\|_{\infty}+
\left\|g_{f_{i_0},N}-g_{f,N}\right\|_{\infty}+
\left\|\int_G f_{i_0}-\int_G f\right\|_{\infty}\\
&\le \left\|g_{f_{i_0},N}-\int_G f_{i_0}\right\|_{\infty}+
\left\|g_{(f_{i_0}-f),N}\right\|_{\infty}+
\left\|f_{i_0}-f\right\|_{\infty}\\
&\le \left\|g_{f_{i_0},N}-\int_G f_{i_0}\right\|_{\infty}+
2\left\|f_{i_0}-f\right\|_{\infty}\\
&< \frac{\varepsilon }{3}+\frac{2\varepsilon }{3}=\varepsilon.
\end{align*}
Hence $S$ is a closed subspace.

\noindent\textbf{Claim}: If $[\pi ]\in \hat{G}\setminus \{[1]\}$, then $\pi_{i,j} \in S$.

Firstly, we will show that $\int_G \pi$ is the zero matrix. Hence $\int_G \pi_{i,j}=0$.

By the property of Haar measure on compact groups that it is two-sided invariant, we have
\begin{align*}
(\int_G \pi)&= (\int_G \pi(x) dx)\\
			&= (\int_G \pi(xy) dx) = (\int_G \pi(x) dx)\pi(y) = (\int_G \pi)\pi(y)
\end{align*}
and
\begin{align*}
(\int_G \pi)&= (\int_G \pi(x) dx)\\
			&= (\int_G \pi(yx) dx) = \pi(y)(\int_G \pi(x) dx) = \pi(y)(\int_G \pi)
\end{align*}
for any $y\in G$. So $\int_G \pi$ commutes with all $\pi(G)$, by Schur's Lemma \cite[Lemma 3.5]{A_Course_in_Abstract_Harmonic_Analysis}, $\int_G \pi = c I_{H_{\pi}}$ for some $c\in \mathbb{C}$. If $c\neq 0$, then $(\int_G \pi) = \pi(y)(\int_G \pi)$ implies $I_{H_{\pi}}=\pi(y)$ for all $y\in G$, which contradicts to the assumption that $[\pi ]\neq [1]$. Thus $\int_G \pi$ is zero matrix.

Secondly, let $\left\|\cdot\right\|_M$ denote the operator norm. We will show that $\left\|\left\|g_{\pi, N}\right\|_M\right\|_{\infty}\to 0$ as $N\to \infty$. In particular, $\left\|g_{\pi_{i,j}, N}-\int_G \pi_{i,j}\right\|_{\infty}=\left\|g_{\pi_{i,j}, N}\right\|_{\infty}\to 0$ as $N\to \infty$. Hence $\pi_{i,j}\in S$.

\begin{align*}
\left\|\left\|g_{\pi, N}\right\|_M\right\|_{\infty}
&=\frac{1}{N} \left\|\left\|\sum\limits_{n=1}^{N-1}\pi(\cdot a^{-n})\right\|_M\right\|_{\infty}\\
&=\frac{1}{N} \left\|\left\|\sum\limits_{n=1}^{N-1}\pi(\cdot )\pi(a^{-n})\right\|_M\right\|_{\infty}\\
&\le\frac{1}{N} \left\|\left\|\pi(\cdot )\right\|_M\left\|\sum\limits_{n=1}^{N-1}\pi(a)^{-n}\right\|_M\right\|_{\infty}\\
&=\frac{1}{N} \left\|\sum\limits_{n=1}^{N-1}\pi(a)^{-n}\right\|_M\text{ ($\pi(x)$ is an unitary matrix for each $x$.)}\\
&\le\frac{1}{N} \left\|(I_{H_{\pi}}-\pi(a))^{-1}\right\|_M\left\|I_{H_{\pi}}-\pi(a)^N\right\|_M\\
&\le\frac{2}{N} \left\|(I_{H_{\pi}}-\pi(a))^{-1}\right\|_M\to 0 \text{  as $N\to \infty.$}
\end{align*}

\end{proof}

\begin{definition}
\label{deftransunifequi}
Let $G$ be a compact Hausdorff group. We say a sequence $\{x_k\}_{k\in \mathbb{N}}$ is a homogeneous equidistribution w.r.t. $A$, if $A$ is a collection of subsets of $G$ and 
\[\sup_{x\in G}\left|Dens_{N,K,\{x_k\}_{k\in \mathbb{N}}}(x)-|K|\right|\to 0 \text{ as } N\to \infty \text{ for any } K\in A,\]
where $Dens_{N,K,\{x_k\}_{k\in \mathbb{N}}}(x):=dens_N(x^{-1}K,\{x_k\}_{k\in \mathbb{N}})$ and
\[dens_N(K,\{x_k\}_{k\in \mathbb{N}}):=\frac{\#\{x_k\in K|k<N\}}{N}=\frac{1}{N}\sum\limits_{k=1}^{N-1}\chi_K(x_k).\]
\end{definition}

\begin{lemma}
\label{lemmatransunifequiofelement}
Let $G$ be a compact Hausdorff group, and $a$ be an element in $G$ such that $\pi(a)$ do not have nontrivial fixed points in $H_{\pi}$ for all $[\pi ]\in \hat{G}\setminus \{[1]\}$. Then $\{a^{-n}\}_{n\in \mathbb{N}}$ is a homogeneous equidistribution w.r.t. $\mathcal{A}$.
\end{lemma}

\begin{proof}
We divide the proof into two parts. For the first part, we will prove that $\{a^{-n}\}_{n\in \mathbb{N}}$ is a homogeneous equidistribution w.r.t. the form (1) in $\mathcal{A}$. (We have defined $\mathcal{A}$ in Section \ref{sec:notation}.)

For any $\varepsilon>0$ and $U$ open in $G$ with $\partial U$ being measure $0$, define $f_{U,\varepsilon}^+$ and $f_{U,\varepsilon}^-$ to be some continuous functions which satisfy the following:

\begin{enumerate}[{(1)}]
\item
Both image of $f_{U,\varepsilon}^+$ and $f_{U,\varepsilon}^-$ lie in $[0,1]$.
\item
$f_{U,\varepsilon}^+|_{\overline{U}}\equiv 1$ and $f_{U,\varepsilon}^+|_{V^c}\equiv 0$, where $V$ is an open neighborhood of $\overline{U}$ such that $|V\setminus U|=|V\setminus \overline{U}|<\varepsilon/2$.
\item
$f_{U,\varepsilon}^-|_{K}\equiv 1$ and $f_{U,\varepsilon}^-|_{U^c}\equiv 0$, where $K$ is a compact subset of $U$ such that $|U\setminus K|=|\overline{U}\setminus K|<\varepsilon/2$.
\end{enumerate}

The existence of $f_{U,\varepsilon}^+$ and $f_{U,\varepsilon}^-$ follow by the Uryshon's lemma (It works, since $G$ is compact Hausdorff, hence normal.) and the regularity of Haar measure. From the construction, we immediatly have
\[\chi_{K} \le f_{U,\varepsilon}^-\le \chi_{\widetilde{U}} \le f_{U,\varepsilon}^+ \le \chi_{V},\]
hence (use the notation in the proof of Lemma \ref{uniformlyconverge})
\[g_{f_{U,\varepsilon}^-,N}\le dens_{N}(x^{-1}\widetilde{U},\{a^{-n}\}_{n\in \mathbb{N}}) \le g_{f_{U,\varepsilon}^+,N}\text{ for any $N$.}\]
By Lemma \ref{uniformlyconverge}, there exists $N'$ independent of $x$ such that 
\begin{align*}
|\widetilde{U}|-\varepsilon 
&= \int_G \chi_{\widetilde{U}}-\varepsilon \\
&\le \int_G f_{U,\varepsilon}^- -\varepsilon/2 \\
&< g_{f_{U,\varepsilon}^-,N}\\
&\le dens_{N}(x^{-1}\widetilde{U},\{a^{-n}\}_{n\in \mathbb{N}})
= Dens_{N,\widetilde{U},\{a^{-n}\}_{n\in \mathbb{N}}}(x)\\
&\le g_{f_{U,\varepsilon}^+,N}\\
&< \int_G f_{U,\varepsilon}^+ +\varepsilon/2 \\
&\le \int_G \chi_{\widetilde{U}}+\varepsilon \\
&=|\widetilde{U}|+\varepsilon 
\end{align*}
for all $N>N'$. Then we have 
\[\sup_{x\in G}\left|Dens_{N,\widetilde{U},\{a^{-n}\}_{n\in \mathbb{N}}}(x)-|\widetilde{U}| \right|<\varepsilon.\]

For the sencond part, we will consider the additivity of the function $Dens$.

Let $X,Y$ be disjoint subsets of $G$. If any two of the followings below hold, then so does the third.
\begin{enumerate}[{(1)}]
\item
$\{a^{-n}\}_{n\in \mathbb{N}}$ is a homogeneous equidistribution w.r.t. $\{X\}$.
\item
$\{a^{-n}\}_{n\in \mathbb{N}}$ is a homogeneous equidistribution w.r.t. $\{Y\}$.
\item
$\{a^{-n}\}_{n\in \mathbb{N}}$ is a homogeneous equidistribution w.r.t. $\{X\cup Y\}$.
\end{enumerate}
This follows by the equations 
\[Dens_{N,X\cup Y,\{a^{-n}\}_{n\in \mathbb{N}}}=Dens_{N,X,\{a^{-n}\}_{n\in \mathbb{N}}}+Dens_{N,Y,\{a^{-n}\}_{n\in \mathbb{N}}}\]
and 
\[|X\cup Y|=|X|+|Y|.\]

For the form (2) in $\mathcal{A}$. Let $S=\widetilde{U}\setminus U$, where $\widetilde{U}$ has the form (1). We can write $\widetilde{U}=S\cup U$ as a disjoint union. Hence $\{a^{-n}\}_{n\in \mathbb{N}}$ is a homogeneous equidistribution w.r.t. $\{S\}$, since both $\widetilde{U}$ and $U$ are form (1).

Similar to the form (3) in $\mathcal{A}$. We can write it as the finite disjoint union of the form (1) and (2). So $\{a^{-n}\}_{n\in \mathbb{N}}$ is a homogeneous equidistribution w.r.t. the form (3) in $\mathcal{A}$.

Hence the lemma has been proved.

\end{proof}

\begin{corollary}
Let $G$ be a compact Hausdorff group and $a\in G$. Then the following are equivalent.
\begin{enumerate}[{(1)}]
\item
$G=\overline{<a>}$.
\item
For any $\pi \in \hat{G}\setminus \{[1]\}$, $\pi(a)$ does not have nontrivial fixed points in $H_{\pi}$ for all $[\pi ]\in \hat{G}\setminus \{[1]\}$.
\end{enumerate}
\end{corollary}

\begin{proof}
(1$\Rightarrow $2)
Since $G=\overline{<a>}$ is abelian, $\hat{G}$ is its daul group. And we know that $\dim H_{\pi} = 1$ for all $\pi\in \hat{G}$. So for any $\pi \in \hat{G}\setminus \{1\}$, $1-\pi(a)\neq 0$ is invertible, since if $\pi(a)=1$, we have $\pi(a^n)=1$ for all $n\in \mathbb{Z}$, a contradiction as $\pi\equiv 1$.

(2$\Rightarrow $1)
By Lemma \ref{lemmatransunifequiofelement}, $\{a^{-n}\}_{n\in \mathbb{N}}$ is a homogeneous equidistribution w.r.t. $\mathcal{A}$. This implies that $\{a^{-n}\}_{n\in \mathbb{N}}$ is dense in $G$ (it's sufficient to say any that non-empty open set in $G$ contains an element of form (1) in $\mathcal{A}$, see \cite[Chapter 3, example 1.2.]{Uniform_Distribution_of_Sequences}), hence $G=\overline{<a>}$.
\end{proof}

\begin{corollary}
\label{noncyclic}
Let $G$ be a finite group. Then the following are equivalent.
\begin{enumerate}[{(1)}]
\item
$G$ is non-cyclic.
\item
For any $a \in G$, there exists $\pi \in \hat{G}\setminus \{[1]\}$ such that $\pi(a)$ has an eigenvalue $1$.
\end{enumerate}
\end{corollary}

\section{A necessary condition of hypercyclic weighted translation operators} 
\label{sec:a_necessary_condition_of_hypercyclic_weighted_translation_operators}

\begin{theorem}
\label{conditionofwsoncompactgp}
Let $G$ be a compact Hausdorff group, and let $a$ be an element in $G$ such that $G=\overline{<a>}$. If $T_{a,w}$ is hypercyclic on $L^p(G)$ for some $p\in [1,\infty )$, then 
\[\int_G \ln w = 0.\]
\end{theorem}

\begin{proof}
First, we will observe some facts for the step functions.

Let $\phi> 0$ be a strictly positive $\mathcal{A}$-step function (i.e $\phi$ can write as $\sum\limits_{i=1}^k \alpha_i \chi_{E_i}$, where $\alpha_i>0$ and $E_i\in \mathcal{A}$ are disjoint). Set $m_i:=|E_i|$. By Lemma \ref{lemmatransunifequiofelement}, for any $\varepsilon >0$, there exists $N'$ independent of $x$ such that 
\[m_i-\varepsilon< \frac{1}{N} \#\{xa^{-n}\in E_i|n<N\} <m_i+\varepsilon\]
for all $x\in G$ and $N>N'$ and $i=1 \cdots k$.
Then
\[\prod \limits_{i=1}^k \alpha_i^{m_i-\varepsilon}
< \left (\prod \limits_{i=1}^k \alpha_i^{\#\{xa^{-n}\in E_i|n<N\}} \right )^{\frac{1}{N}}
< \prod \limits_{i=1}^k \alpha_i^{m_i+\varepsilon}
.\]
Hence
\[\prod \limits_{i=1}^k \alpha_i^{m_i-\varepsilon}
< \left ( \phi_N \right )^{\frac{1}{N}}
< \prod \limits_{i=1}^k \alpha_i^{m_i+\varepsilon}
,\]
where $\phi_N (x):= \prod \limits_{n=0}^{N-1} \phi (xa^{-n})$, since $\prod \limits_{n=0}^{N-1} \phi (xa^{-n}) = \prod \limits_{i=1}^k \alpha_i^{\#\{xa^{-n}\in E_i|n<N\}}$.

Now we are going to claim two things.
Let $\phi$ be a strictly positive $\mathcal{A}$-step function, then
\begin{enumerate}[{(1)}]
\item
$w\geq\phi$ implies $\int_G \ln\phi\le 0$
\item
$w\le\phi$ implies $\int_G \ln\phi\geq 0$.
\end{enumerate}

Since 
$\int_G \ln\phi = \sum\limits_{i=1}^k m_i\ln \alpha_i = \ln \left ( \prod \limits_{i=1}^k \alpha_i^{m_i}\right )$, so $\int_G \ln\phi\le 0 \Leftrightarrow \prod \limits_{i=1}^k \alpha_i^{m_i}\le 1$ and $\int_G \ln\phi\geq 0 \Leftrightarrow \prod \limits_{i=1}^k \alpha_i^{m_i}\geq 1$.

For (1). Suppose $\prod \limits_{i=1}^k \alpha_i^{m_i}> 1$ but $w\geq\phi$. We can choose $\varepsilon $ small enough such that $1< \left ( \phi_N \right )^{\frac{1}{N}}$. Then $1< \phi_N \le w_N$, where $w_N(x):= \prod \limits_{n=0}^{N-1} w(xa^{-n})$. But $T_{a,w}^N=T_{a^N,w_N}$ is also hypercyclic \cite{Hypercyclic_and_cyclic_vectors}, contradict to the fact that a weighted translation operator with the wieght greater then $1$ will never be a hypercyclic operator.

Simiarly, for (2). Suppose $\prod \limits_{i=1}^k \alpha_i^{m_i}< 1$, but $w\le\phi$. We can choose $\varepsilon $ small enough such that $1> w_N$, contradict to the fact that a weighted translation operator with the wieght smaller then $1$ will never be a hypercyclic operator.

Note that we write $\Phi:=\ln \phi$ and $W:=\ln w$, then $\phi \le w \Leftrightarrow \Phi \le W$ and $\phi \geq w \Leftrightarrow \Phi \geq W$. Moreover, $\Phi$ is still a $\mathcal{A}$-step function and any $\mathcal{A}$-step function is the form $\ln \phi$.

Finally, we will prove $\int_G \ln w = \int_G W = 0$ by the following.
\begin{enumerate}[{(i)}]
\item
If the statement ``$W\geq\Phi$ implies $\int_G \Phi\le 0$'' holds for all $\Phi$, then $\int_G W\le 0$.
\item
If the statement ``$W\le\Phi$ implies $\int_G \Phi\geq 0$'' holds for all $\Phi$, then $\int_G W\geq 0$.
\end{enumerate}
To prove this, it is sufficient to show that the $\mathcal{A}$-step functions are enough to approximate the continuous function $W$ from above (from below) in $\|\cdot\|_{\infty}$ norm. 

Consider the sets $B_{\beta}:=\{x|W(x)=\beta\}$. If there are uncountable many $\beta$ such that $B_{\beta}$ has positive measure. Then write $G$ as the disjoint union $\cup_{\beta\in\mathbb{R}}B_{\beta}$. We get $G$ has infinte measure, which is a contradiction. Hence there are at most countably many $\beta$ such that $B_{\beta}$ has positive measure.

From the consideration of $B_{\beta}$, we see that apart from at most countably many exceptions, the sets $F_{\beta}:= \{x|W(x)\geq \beta\}$ has null measure boundary.

Let $\varepsilon>0$ and $\inf Im W =\beta_0< \beta_1 \cdots < \beta_{\ell-1}<\beta_\ell= \sup Im W$ ($Im W$ means the image of $W$) with $\beta_{i+1}-\beta_i<\varepsilon$ and $E_i:=F_{\beta_i}$ having an null measure boundary for each $0\le i<\ell$.
The existence of such $\beta_i$ follows from the conclusion above that there are at most countably many illegal numbers that $\beta_i$ could choose to be.
Set $\Phi:=\beta_0+\sum\limits_{i=0}^{\ell-1}(\beta_{i+1}-\beta_i)\chi_{E_i}$ as an $\mathcal{A}$-step function.
Then for each $x\in X$, there exists an integer $j$, $0 \le j < n$ with $\beta_j\le W(x)<\beta_{j+1}$ such that
\begin{align*}
\Phi(x)-W(x)
&=\beta_0+\sum\limits_{i=0}^{\ell-1}(\beta_{i+1}-\beta_i)\chi_{E_i}(x)-W(x)\\
&= \beta_0+\sum\limits_{i=0}^{j}(\beta_{i+1}-\beta_i)\chi_{E_i}(x)-W(x)\\
&= \beta_{j+1}-W(x)\\
&\le \beta_{j+1}-\beta_j\le\varepsilon.
\end{align*}
Also, from the equation $\Phi(x)-W(x)= \beta_{j+1}-W(x)$, we see that $\Phi\ge W$.
Hence we have
\begin{align*}
\int_G W&= \int_G (W-\Phi) +\int_G \Phi\\
&\geq -\varepsilon +\int_G \Phi \\
&\geq -\varepsilon.
\end{align*}
So $\int_G W\geq 0$.

Similarly, we can find $\|\Phi-W\|_{\infty}<\varepsilon$ and $W\geq\Phi$. Hence $\int_G W\le 0$.
\end{proof}
\begin{remark}
In fact, the condition ``$T_{a,w}$ is hypercyclic'' in this theorem can be modified to be more weaker one ``$w_n\ngeqslant 1$ and $w_n\nleqslant 1$ for all $n\in\mathbb{N}$'', since this weaker condition is the only thing which is relative to hypercyclicity that we have used in the proof.
\end{remark}

\begin{example}
Let $S^1=[0,1)$ denote the circle group, $a\in S^1$ be an irrational number. Then the condition of Theorem \ref{conditionofwsoncompactgp} holds.
\end{example}

\begin{remark}
The necessary condition $\int_G \ln w = 0$ for a hypercyclic weighted translation operator sometimes still holds for some non-compact group $G$ (ex: when $G$ is a p-adic field).
\end{remark}

\begin{example}
Let $G=\mathbb{Q}_p$ or $\mathbb{Z}_p$ for some prime number $p$ with $a\in G$. Then if $T_{a,w}$ is hypercyclic on $L^{p'}(G)$ for some $p'\in [1,\infty )$, we also have 
\[\int_G \ln w = 0.\]
\end{example}

\begin{proof}
Let $k=v_p(a)$.
Consider these diagrams. For each coset of $p^k\mathbb{Z}_p$ say $b+p^k\mathbb{Z}_p$

\begin{tikzpicture}
\node (A) {$L^{p'}(G)$};
\node (B) [right = 4em of A] {$L^{p'}(G)$};
\node (C) [below = of A] {$L^{p'}(G)$};
\node (D) [right = 4em of C] {$L^{p'}(G)$};
\node (E) [below = of C] {$L^{p'}(p^k\mathbb{Z}_p)$};
\node (F) [right = 2.8em of E] {$L^{p'}(p^k\mathbb{Z}_p)$};
\draw [thick, ->] (A) -- node [anchor = south] {$T_{a,w}$} (B);
\draw [thick, ->] (A) -- node [xshift = 3.3em, anchor = west] {$\circlearrowright$} 
 node [anchor = east] {$T_{-b}$} (C);
\draw [thick, ->] (B) --node [anchor = west] {$T_{-b}$} (D);
\draw [thick, ->] (C) --node [xshift = .3em, yshift = -.2em, anchor = north] {$T_{a,(T_{-b} w)}$} (D);
\draw [thick, ->] (C) -- node [xshift = 3.3em, yshift = -.2em, anchor = west] {$\circlearrowright$}
 node [anchor = east] {$Res$} (E);
\draw [thick, ->] (D) --node [anchor = west] {$Res$} (F);
\draw [thick, ->] (E) --node [xshift = .1em, yshift = -.5em, anchor = north] {$T_{a,\left((T_{-b} w)|_{p^k\mathbb{Z}_p}\right)}$} (F);
\end{tikzpicture},

where $Res$ means the restriction map.

If $T_{a,w}$ is hypercyclic on $L^{p'}(G)$, then so is $T_{a,\left((T_{-b} w)|_{p^k\mathbb{Z}_p}\right)}$ on $L^{p'}(p^k\mathbb{Z}_p)$. And since $p^k\mathbb{Z}_p$ is a compact abelian group with $\overline{<a>}=p^k\mathbb{Z}_p$, by Theorem \ref{conditionofwsoncompactgp}, we have 
\[\int_{p^k\mathbb{Z}_p} \ln (T_{-b} w) = 0.\]
But
\[0 = \int_{p^k\mathbb{Z}_p} \ln (T_{-b} w) =\int_{p^k\mathbb{Z}_p} \ln w(x+b) dx = \int_{b+p^k\mathbb{Z}_p} \ln w(x) dx.\]
In short, if $T_{a,w}$ is hypercyclic on $L^{p'}(G)$, then for each coset of $p^k\mathbb{Z}_p$ say $b+p^k\mathbb{Z}_p$, we have 
\[\int_{b+p^k\mathbb{Z}_p} \ln w = 0.\]
Hence \[\int_G \ln w = 0.\]
\end{proof}

\section{Some properties of hypercyclic weighted translations on p-adic} 
\label{sec:some_properties_of_hypercyclic_weighted_translations_on_p_adic}

\begin{theorem}
\label{conditionofws}
Let $T_{a,w}$ be a weighted translation operator on $L^{p'}(\mathbb{Z}_p)$ for some prime number $p$ and some $p'\in[1,\infty)$. Then for any given $x'\in G$ and $ n\in \mathbb{Z}\setminus\{0\}$.
Define the sets 
\[U_{w,n,x'}:=\{x\in B_{\leq |na|_p}(x')|w_n(x)>1\}\]
and
\[L_{w,n,x'}:=\{x\in B_{\leq |na|_p}(x')|w_n(x)<1\},\]
where $B_{\leq r}(x'):=\{x||x-x'|_p\leq r\}$.
If $T_{a,w}$ is hypercyclic, then
\[U_{w,n,x'}\neq \varnothing\]
and
\[L_{w,n,x'}\neq \varnothing.\]
\end{theorem}

\begin{proof}
Let $T_{a,w}$ be hypercyclic. Consider these two diagrams

\begin{tikzpicture}
\node (A) {$L^{p'}(\mathbb{Z}_p)$};
\node (B) [right = of A] {$L^{p'}(\mathbb{Z}_p)$};
\node (C) [below = of A] {$L^{p'}(\mathbb{Z}_p)$};
\node (D) [right = of C] {$L^{p'}(\mathbb{Z}_p)$};
\draw [thick, ->] (A) -- node [anchor = south] {$T_{a,w}$} (B);
\draw [thick, ->] (A) -- node [xshift = 2.7em, anchor = west] {$\circlearrowright$} 
 node [anchor = east] {$T_{-x'}$} (C);
\draw [thick, ->] (B) --node [anchor = west] {$T_{-x'}$} (D);
\draw [thick, ->] (C) --node [xshift = .3em, yshift = -.2em, anchor = north] {$T_{a,(T_{-x'} w)}$} (D);
\end{tikzpicture},
\begin{tikzpicture}
\node (A) {$L^{p'}(\mathbb{Z}_p)$};
\node (B) [right = of A] {$L^{p'}(\mathbb{Z}_p)$};
\node (C) [below = of A] {$L^{p'}(\mathbb{Z}_p)$};
\node (D) [right = of C] {$L^{p'}(\mathbb{Z}_p)$};
\draw [thick, ->] (A) -- node [anchor = south] {$T_{a,(T_{-x'} w)}^n$} (B);
\draw [thick, ->] (A) -- node [xshift = 2.7em, anchor = west] {$\circlearrowright$} 
 node [anchor = east] {$M_p^{v_p(na)}$} (C);
\draw [thick, ->] (B) --node [anchor = west] {$M_p^{v_p(na)}$} (D);
\draw [thick, ->] (C) --node [anchor = north] {$T$} (D);
\end{tikzpicture},

where the operator $T:=T_{a^{(n)},w^{(n)}}$, $a^{(n)}:=\frac{na}{p^{v_p(na)}}$, $w^{(n)}(x):=w_n(p^{v_p(na)}x+x')$.

Since the diagrams commute, so $T$ must be hypercyclic too. Then its weight function must not be identically smaller than $1$. Hence
\begin{align*} 
\varnothing &\neq\{x|w^{(n)}(x)>1\}\\
&=\{x|w_n(p^{v_p(na)}x+x')>1\}.
\end{align*}
This implies
\begin{align*}
\varnothing &\neq\{x\in B_{\le |na|_p}(0)|w_n(x+x')>1\}+x'\\
&=\{x+x'\in B_{\le |na|_p}(x')|w_n(x+x')>1\}\\
&=\{x\in B_{\le |na|_p}(x')|w_n(x)>1\}\\
&=U_{w,n,x'}.
\end{align*}
Similar to the case of $L_{w,n,x'}$.
\end{proof}

\begin{corollary}
\label{ulczp}
If $w$ is a locally constant weight for the weighted translation operator $T_{a,w}$, then $T_{a,w}$ is not hypercyclic on $L^{p'}(\mathbb{Z}_p)$ for any $p'\in[1,\infty)$.
\end{corollary}

\begin{proof}
Let $w$ be a locally constant function. Since $G$ is compact, 
we see that $w$ is a u.l.c.(uniformly locally constant function). (See \cite[Lemma 3.1.1, p.178]{A_Course_in_p-adic_Analysis}).

Since any ball in $\mathbb{Z}_p$ is a coset of subgroup $p^k\mathbb{Z}_p$ for some $k\in \mathbb{N}$, so $w$ can be viewed as a composition

\begin{tikzpicture}
\node (A) {$\mathbb{Z}_p$};
\node (B) [right = of A] {$\mathbb{Z}_p/p^k\mathbb{Z}_p$};
\node (D) [below = of B] {$(0,\infty)$};
\draw [thick, ->] (A) -- node [anchor = south] {$\pi$} (B);
\draw [thick, ->] (A) -- node [xshift = .5em, yshift = .63em, anchor = west] {$\circlearrowright$} 
 node [yshift = -.2em, anchor = east] {$w$} (D);
\draw [thick, ->] (B) --node [yshift = .3em, anchor = west] {$\overline{w}$} (D);
\end{tikzpicture},

where $\pi$ is the natural quotient map.
Then we call $w$ is a k-u.l.c..

Observe that $T_a w$ is also a k-u.l.c. and the multiplication of two k-u.l.c. are also k-u.l.c.. These imply that $w_n$ are k-u.l.c. for all $n\in \mathbb{N}$.

But by the previous theorem, 
\[U_{w,n,0}=\{x|w_n|_{B_{\le |na|_p}(0)}(x)>1\}\neq \varnothing \]
and
\[L_{w,n,0}=\{x|w_n|_{B_{\le |na|_p}(0)}(x)<1\}\neq \varnothing \]
for any $n\in \mathbb{N}$.

To make the contradiction, we just choose $n= p^k$. Then $w_n|_{B_{\le |na|_p}(0)}$ is a constant function, which cannot satisfy both of them simultaneously.
\end{proof}

\begin{corollary}
\label{ulcqp}
If $w$ is a locally constant weight for the weighted translation operator $T_{a,w}$, then $T_{a,w}$ is not hypercyclic on $L^{p'}(\mathbb{Q}_p)$ for any $p'\in[1,\infty)$.
\end{corollary}
\begin{proof}
Consider the diagrams

\begin{tikzpicture}
\node (A) {$L^{p'}(\mathbb{Q}_p)$};
\node (B) [right = of A] {$L^{p'}(\mathbb{Q}_p)$};
\node (C) [below = of A] {$L^{p'}(\mathbb{Q}_p)$};
\node (D) [right = of C] {$L^{p'}(\mathbb{Q}_p)$};
\node (E) [below = of C] {$L^{p'}(\mathbb{Z}_p)$};
\node (F) [right = of E] {$L^{p'}(\mathbb{Z}_p)$};
\draw [thick, ->] (A) -- node [anchor = south] {$T_{a,w}$} (B);
\draw [thick, ->] (A) -- node [xshift = 2.7em, anchor = west] {$\circlearrowright$} 
 node [anchor = east] {$M_a$} (C);
\draw [thick, ->] (B) --node [anchor = west] {$M_a$} (D);
\draw [thick, ->] (C) --node [xshift = .3em, yshift = -.2em, anchor = north] {$T_{1,(M_a w)}$} (D);
\draw [thick, ->] (C) -- node [xshift = 2.7em, yshift = -.2em, anchor = west] {$\circlearrowright$}
 node [anchor = east] {$Res$} (E);
\draw [thick, ->] (D) --node [anchor = west] {$Res$} (F);
\draw [thick, ->] (E) --node [xshift = .8em, yshift = -.5em, anchor = north] {$T_{1,\left((M_a w)|_{\mathbb{Z}_p}\right)}$} (F);
\end{tikzpicture},

where $Res$ means the restriction map.

So if $T_{a,w}$ is hypercyclic on $L^{p'}(\mathbb{Q}_p)$, then $T_{1,\left((M_a w)|_{\mathbb{Z}_p}\right)}$ is also hypercyclic on $L^{p'}(\mathbb{Z}_p)$. But $(M_a w)|_{\mathbb{Z}_p}$ is also a locally constant function, which contradicts to the Corollary \ref{ulczp}.
\end{proof}

\begin{theorem}
\label{p_adic_main}
If $T_{a,w}$ is hypercyclic on $L^{p'}(G)$ where $G=\mathbb{Q}_p \text{ or } \mathbb{Z}_p$ for some prime $p$ and $p'\in [1,\infty)$, then
$$\int_G w = 0$$
and $w$ must not be a locally constant function.
\end{theorem}



\bibliographystyle{abbrv}
\bibliography{Weighted_translation_operators_on_locally_compact_group} 

\vspace{.1in}
\end{document}